\documentclass[11pt]{amsart}

\usepackage{style}

\usepackage[pagewise]{lineno}

\begin{document}

\title{Volumes of line bundles as limits on generically nonreduced schemes}

\author{Roberto N\'u\~nez}

\address{Roberto N\'u\~nez, Department of Mathematics,
University of Missouri, Columbia, MO 65211, USA}
\email{rnhvb@mail.missouri.edu}

\begin{abstract}
The volume of a line bundle is defined in terms of a limsup. It is a fundamental question whether this limsup is a limit. It has been shown that this is always the case on generically reduced schemes. We show that volumes are limits in two classes of schemes that are not necessarily generically reduced: codimension one subschemes of projective varieties such that their components of maximal dimension contain normal points and projective schemes whose nilradical squared equals zero. 
\end{abstract}


\maketitle

\section{Introduction}\label{Intro}

Let $X$ be a proper scheme of dimension $d$ over a field $k$. Let $\mathscr{L}$ be a line bundle on $X$. The \textit{volume} of $\mathscr{L}$ is defined as follows:

$$  \vol(\mathscr{L}) \coloneqq \limsup \limits_{n \to \infty} \dfrac{\dim_k\coh^0(X,\mathscr{L}^n)}{n^d/d!}.$$

\noindent This invariant is defined in \cite[Definition 2.2.31]{Laz01}, where several interesting properties are derived. A fundamental problem to study is whether the limsup in the definition of volume is a limit. Whenever this is the case, we will say that the volume of $\mathscr{L}$ exists as a limit. In this direction, we have the following theorem.

\begin{theorem}[{{\cite[Theorem 10.7]{Cut01}}}]\label{Theorem14a} Suppose that $X$ is a proper scheme of dimension $d$ over a  field $k$. Let $\mathcal{N}$ be the nilradical of $X$. Assume that $\dim(\Supp(\mathcal{N}))< d$. Let $\mathscr L$ be a line bundle on $X$.  Then, the limit

$$ \lim_{n\rightarrow \infty}\frac{\dim_k \coh^0(X,\mathscr L^n)}{n^d} $$

\noindent exists. That is,  $\vol(\mathscr L)$ exists as a limit.
\end{theorem}

An example of a line bundle on a projective variety where the limit in Theorem \ref{Theorem14a} is an irrational number is given in Example 4 of Section 7 of \cite{CS}.

\cref{Theorem14a} has been proven by Lazarsfeld and Musta\c{t}\u{a} \cite{LazMus} for big line bundles and by Kaveh and Khovanskii \cite{KavKov} for graded linear series (see definition below), both under the assumption that $k$ is an algebraically closed field. These proofs use a method invented by Okounkov \cite{Ok} which reduces the situation to a problem of
counting points in an integral semigroup. In these proofs, the assumption that the base field is algebraically closed is necessary. The proof of \cref{Theorem14a} also uses Okounkov's technique. 

Recall that $X$ is a proper scheme over a field $k$. To any line bundle $\mathscr{L}$ on $X$ we can associate a \textit{section ring} 

$$ R(\mathscr{L}) \coloneqq \displaystyle \bigoplus_{n \geq 0} \coh^{0}(X,\mathscr{L}^{n}). $$

\noindent Under the inclusion $k \subset \mathcal{O}_{X}(X)$, this is a graded $k$-algebra. One is naturally led to consider graded subalgebras of this section ring. For example, fix a line bundle $\mathscr{L}$ and let $X_{red}$ be the reduced scheme associated to $X$. For all $n \in  \mathbb{N}$, the natural morphism $\mathcal{O}_{X} \rightarrow \mathcal{O}_{X_{red}}$ can be tensored with $\mathscr{L}^n$ to induce maps of vector spaces

$$ \Phi_{n}\colon \coh^{0}(X,\mathscr{L}^{n}) \rightarrow \coh^{0}(X_{red}, \left(\mathscr{L}|_{X_{red}}\right)^{n}). $$

\noindent Since the global section functor is not right-exact, it is not in general true that the $k$-algebra

$$ L \coloneqq \displaystyle \bigoplus_{n \geq 0} \Phi_{n}\left(\coh^{0}(X,\mathscr{L}^{n})\right)$$

\noindent is equal to the section ring associated to $\mathscr{L}|_{X_{red}}$ on $X_{red}$. It is however the case that $L$ is a $k$-subalgebra of $R(\mathscr{L}|_{X_{red}})$.

A \textit{graded linear series} on $X$, often simply called a \textit{linear series}, is a graded $k$-subalgebra $L=\oplus_{n\ge 0}L_n$ of a section ring $R(\mathscr{L})$ of some line bundle $\mathscr L$ on $X$. The linear series $L$ is said to be \textit{complete} if $L=R(\mathscr{L})$. We define the \textit{volume} of a linear series $L$ on $X$ as follows:

$$ \vol(L) \coloneqq \limsup_{n \to \infty}\dfrac{\dim_{k}L_n}{n^d/d!}.$$

\noindent Notice that when $L$ is a complete linear series, we recover the definition of volume of the corresponding line bundle.

The following theorem is a consequence of \cite[Theorem 1.4]{Cut01} and \cite[Theorem 10.3]{Cut01}. 

\begin{theorem}\label{TheoremI1} Suppose that $X$ is a $d$-dimensional projective scheme over a field $k$ with 
$d>0$. Let $\mathcal N$ be the nilradical of $X$. Then, the following are equivalent:
\begin{enumerate}
\item[1)] For every  graded linear series $L$ on $X$ there exists a positive integer $r$ (depending on $L$) such that the limit
$$
\lim_{n\rightarrow\infty}\frac{\dim_k L_{rn}}{n^d}
$$
exists.
\item[2)] $\dim(\Supp(\mathcal{N}))< d$.
\end{enumerate}
\end{theorem}

The implication $1) \implies 2)$ of \cref{TheoremI1} is proven by constructing, on any given projective scheme not satisfying  $2)$, a linear series $L$ such that the limit in $1)$ does not exist (for any $r$). In particular, the volume $\vol(L)$ does not exist as a limit. This linear series, however, is not a complete linear series. Therefore, the following question was raised in \cite{CN}.

\begin{question} \label{question}
Let $X$ be a projective scheme over a field and let $\mathscr L$ be a line bundle on $X$. Does the volume of $\mathscr L$ exist as a limit?
\end{question}

This question has a positive answer whenever $\dim(\Supp(\mathcal{N}))< d$ by \cref{Theorem14a}. In \cite{CN}, \cref{question} was answered affirmatively for codimension one subschemes of nonsingular varieties. In this paper, we also give a positive answer in the following two situations.

\begin{theorem} \label{thm: existence_limits_codimension1}

Let $X$ be a $(d+1)$-dimensional projective variety over an arbitrary field. Let $Y \subset X$ be a subscheme of codimension one. Assume that no component of $Y$ of maximal dimension is contained in the nonnormal locus of $X$. Let $\mathscr{L}$ be a line bundle on $Y$. Then, the limit

$$
\lim_{n\rightarrow \infty}\frac{\dim_k \coh^0(Y,\mathscr L^n)}{n^d}
$$

\noindent exists. That is, $\vol(\mathscr{L})$ exists as a limit.

\end{theorem}

\begin{theorem} \label{thm: nilradical_squared_equals_zero}

Let $X$ be a $d$-dimensional projective scheme over an arbitrary field. Suppose that the nilradical $\mathcal{N}$ of $X$ has the property that $\mathcal{N}^{2}=0$. Let $\mathscr{L}$ be a line bundle on $X$. Then, the limit

$$\lim_{n\rightarrow \infty}\frac{\dim_k \coh^0(X,\mathscr L^n)}{n^d}$$

\noindent exists. That is, $\vol(\mathscr{L})$ exists as a limit.
\end{theorem}

I thank my advisor Dale Cutkosky for his invaluable help during the preparation of this work.

\section{Notation and Terminology} \label{sec: Notation}

Let $X$ be a $d$-dimensional proper scheme over an arbitrary field. The scheme $X$ is said to be a \textit{variety} if it is integral. For any coherent $\mathcal{O}_{X}$-module $\mathcal{F}$, we will use the following notation:

$$h^{i}(X,\mathcal{F}) = \dim_k\coh^i(X,\mathcal{F}).$$

\noindent Let $\mathscr{L}$ and $\mathcal{F}$ represent an invertible sheaf and a coherent sheaf on $X$, respectively. We define the 
\textit{volume of $\mathcal{F}$ with respect to $\mathscr{L}$} to be

$$ \vol_{\mathscr{L}}(\mathcal{F}) \coloneqq \limsup \limits_{n \to \infty} \dfrac{h^{0}\left(X, \mathcal{F} \otimes \mathscr{L}^{n} \right)}{n^{d}/d!}.$$

\noindent We say that the volume of $\mathcal{F}$ with respect to $\mathscr{L}$ exists as a limit if the limsup above is actually a limit. Notice that $\vol_{\mathscr{L}}(\mathcal{O}_{X}) = \vol(\mathscr{L})$ is the volume of the invertible sheaf $\mathscr{L}$ as defined in \cref{Intro}. In fact, \cref{lem: tensoring_invertible_preserves_limit} shows that if $\vol(\mathscr{L})$ exists as a limit, then $\vol_{\mathscr{L}}(\mathcal{F}) = \vol(\mathscr{L})$ whenever $\mathcal{F}$ is invertible.

The \textit{nilradical} of a scheme $X$ is the kernel of the natural surjection $\mathcal{O}_{X} \rightarrow \mathcal{O}_{X_{red}}$, where $X_{red}$ is the reduced scheme associated to $X$. The nilradical of $X$ will be denoted $\mathcal{N}(X)$ or simply $\mathcal{N}$ if there is no danger of confusion.

Following \cite[Lecture 9]{Mum}, we say that a global section $s$ of a line bundle $\mathscr{L}$ on a scheme $X$ is a \textit{non-zero-divisor} if the map $\mathcal{O}_{X} \rightarrow \mathscr{L}$ induced by $s$ is injective. Locally, this means that for any point $p \in X$ and some (any) trivialization $\varphi \colon \mathscr{L}_p \rightarrow \mathcal{O}_{X,p}$, the element $\varphi(s)$ is a non-zero-divisor in the local ring $\mathcal{O}_{X,p}$.

Let $L$ be a linear series on $X$. When $X$ is a variety, we define the \textit{index of $L$} to be natural number

$$ m(L ) \coloneqq [\mathbb{Z}: G],$$

\noindent where $G$ is the group generated by the semi-group $ \{ n \in \mathbb{N} \colon L_n \neq 0 \}$. We define the \textit{index} of a line bundle $\mathscr{L}$ on a variety to be the index of the complete linear series $R(\mathscr{L})$. 

If $X$ is an irreducible scheme, we define the rank of any generically free coherent $\mathcal{O}_{X}$-module $\mathcal{F}$ to be the rank of the free $\mathcal{O}_{X,\xi}$-module $\mathcal{F}_{\xi}$, where $\xi$ is the generic point of $X$.
 
Throughout this paper, the term ``divisor" will always refer to a Cartier divisor. For divisors $A$ and $B$, we will write $A \sim B$ to indicate linear equivalence.

We will often use the term ``line bundle" to refer to invertible sheaves.

\section{Preliminary Results} \label{sec: Preliminaries}

\subsection{Volumes and Exact Sequences}

\begin{lemma}\label{lem: existence_limits_SES}

Let $X$ be a proper scheme over a field $k$. Let $\mathscr{L}$ be an invertible sheaf. Let 

$$ 0 \rightarrow \mathcal{F}_{1} \rightarrow \mathcal{F}_{2} \rightarrow \mathcal{F}_{3} \rightarrow 0$$

\noindent be an exact sequence of coherent $\mathcal{O}_{X}$-modules. Let $\{1,2,3\} = \{i,j,k\}$. Suppose that the volume of $\mathcal{F}_{i}$ with respect to $\mathscr{L}$ exists as a limit and that $\mathcal{F}_{k}$ is supported on a closed subset of dimension strictly less than $\dim(X)$. Then, the volume of $\mathcal{F}_{j}$ with respect to $\mathscr{L}$ exists as a limit as well. Moreover, 

$$ \vol_{\mathscr{L}}(\mathcal{F}_{i})=\vol_{\mathscr{L}}(\mathcal{F}_{j}) .$$

\end{lemma}

The proof of \cref{lem: existence_limits_SES} consists of tensoring the given exact sequence with powers of $\mathscr{L}$, taking cohomology, and using the fact that for any coherent $\mathcal{O}_{X}$-module $\mathcal{F}$ and for all $i$ we have that

$$ \lim \limits_{n \to \infty} \dfrac{h^{i}(X,\mathcal{F} \otimes \mathscr{L}^{n})}{n^{\dim(X)}} =0$$

\noindent whenever $\dim(\Supp(\mathcal{F})) < \dim(X)$ (\cite[Proposition 1.31]{Deb}).

The following corollary shows that the existence of volumes as limits is preserved by morphisms of sheaves that are isomorphisms away from closed sets of dimension smaller than the dimension of the ambient scheme.

\begin{corollary} \label{cor: existence_limits_isomorphism_away_from_closed_sets}

Let $X$ be a proper scheme over a field $k$. Let $\mathscr{L}$ be an invertible sheaf. Let 

$$ 0 \rightarrow \mathcal{K} \rightarrow \mathcal{F} \rightarrow \mathcal{G} \rightarrow \mathcal{C} \rightarrow 0$$

\noindent be an exact sequence of coherent $\mathcal{O}_{X}$-modules. Suppose that $\mathcal{K}$ and $\mathcal{C}$ are supported on closed subsets of dimension strictly less than $\dim(X)$. Then the volume of $\mathcal{F}$ with respect to $\mathscr{L}$ exists as a limit if and only if the volume of $\mathcal{G}$ with respect to $\mathscr{L}$ exists as a limit and then 

$$ \vol_{\mathscr{L}}(\mathcal{F})=\vol_{\mathscr{L}}(\mathcal{G}) .$$

\end{corollary}

\begin{proof}

We break up the given exact sequence into the following two exact sequences:

$$0 \rightarrow \mathcal{K} \rightarrow \mathcal{F} \rightarrow \mathcal{F}/\mathcal{K} \rightarrow 0$$

\noindent and

$$0 \rightarrow \mathcal{F}/\mathcal{K} \rightarrow \mathcal{G} \rightarrow \mathcal{C} \rightarrow 0.$$

\noindent If the volume of $\mathcal{F}$ with respect to $\mathscr{L}$ exists as a limit, we apply \cref{lem: existence_limits_SES} to the first sequence to conclude that the same is true for $\mathcal{F}/\mathcal{K}$ and that $\vol_{\mathscr{L}}(\mathcal{F}/\mathcal{K})=\vol_{\mathscr{L}}(\mathcal{F})$. Then, by \cref{lem: existence_limits_SES} and the second sequence, we see that the volume of $\mathcal{G}$ with respect to $\mathscr{L}$ exists as a limit and $\vol_{\mathscr{L}}(\mathcal{G})=\vol_{\mathscr{L}}(\mathcal{F}/\mathcal{K})=\vol_{\mathscr{L}}(\mathcal{F})$.

An analogous argument shows that if the volume of $\mathcal{G}$ with respect to $\mathscr{L}$ exists as a limit, then the same is true for $\mathcal{F}$ and $\vol_{\mathscr{L}}(\mathcal{F})=\vol_{\mathscr{L}}(\mathcal{G})$.  

\end{proof}

\subsection{Non-zero-divisors, Invertible Sheaves, and Cartier Divisors}

\begin{lemma}\label{lem: existence_nzd}
Let $X$ be a projective scheme over a field.  Let $\mathscr{M}$ be an invertible sheaf on $X$ and let $\mathscr{A}$ be an ample invertible sheaf. Then, there exists an $n_{0} \in \mathbb{N}$ such that $\coh^{0}(X, \mathscr{A}^n \otimes \mathscr{M})$ contains a non-zero-divisor for $n \geq n_0$.
\end{lemma}

\begin{proof}

Let $V_1, \dots, V_s$ be the associated subvarieties of $X$ and choose closed points $x_i \in V_i$. Let $\mathcal I=\oplus_{i=1}^s\mathcal I_{x_i}$, where $\mathcal I_{x_i}$ is the ideal sheaf of the point $x_i$. We have a short exact sequence of $\mathcal O_X$-modules 
$$
0\rightarrow \mathcal I\rightarrow \mathcal O_X\rightarrow \bigoplus_{i=1}^s\mathcal O_{x_i}\rightarrow 0.
$$

\noindent Tensor this exact sequence with $ \mathscr{A}^n \otimes \mathscr{M}$ and take cohomology. By Serre's Vanishing Theorem, for $n\gg 0$, we have an exact sequence
$$	
\coh^0(X, \mathscr{A}^n \otimes \mathscr{M})\rightarrow \bigoplus_{i=1}^sk(x_i)\rightarrow 0,
$$
where $k(x_i)$ is the residue field of the point $x_i$. Thus, there exists a section $s\in \coh^0(X,\mathscr{A}^n \otimes \mathscr{L})$ which does not vanish at any of the $x_i$. It follows that $s$ is a non-zero-divisor on $\mathcal O_X$ since it does not vanish along any of the $V_i$.

\end{proof}

For any scheme $X$, the association $D \to \mathcal{O}_{X}(D)$ gives an injective homomorphism from the group of Cartier divisors modulo linear equivalence to $\Pic(X)$ (see \cite[Corollary II.6.14]{Har}). Nakai \cite{Nak} has shown that if X is a projective scheme over an infinite field, then this homomorphism is an isomorphism. We deduce the known fact that this is still the case for projective schemes over arbitrary fields. 
 
\begin{corollary}\label{cor: invertible_sheaves_from_divisors}

Let $X$ be a projective scheme over a field. Then, for any invertible sheaf $\mathscr{L}$, there exists a Cartier divisor $D$ such that $\mathscr{L} \simeq \mathcal{O}_{X}(D)$. Moreover, under the identification described above, $\Pic(X)$ is generated by effective divisors.

\end{corollary}

\begin{proof}

Choose an ample line bundle $\mathscr{A}$ on X. After perhaps replacing $\mathscr{A}$ with a positive power of itself, we can use \cref{lem: existence_nzd} with $\mathscr{M} = \mathcal{O}_{X}$ to find a non-zero-divisor $t \in \coh^{0}(X, \mathscr{A})$. Then $\mathscr{A} \simeq \mathcal{O}_{X}(H)$, where $H\coloneqq div(t)$.

Again by \cref{lem: existence_nzd}, for some $n \in \mathbb{N}$, we can find a non-zero-divisor $s \in \coh^{0}(X, \mathcal{O}_{X}(nH) \otimes \mathscr{L})$. Thus, $\mathcal{O}_{X}(nH) \otimes \mathscr{L} \simeq \mathcal{O}_{X}(div(s))$. Setting $D = div(s) - nH$, we get that $\mathscr{L} \simeq \mathcal{O}_{X}(D)$.

For the last statement of the corollary, simply notice that $D$ is a difference of effective divisors.

\end{proof}

\subsection{A Lemma on Volume}

We now show that volumes are unaffected by tensoring with invertible sheaves.

\begin{lemma} \label{lem: tensoring_invertible_preserves_limit}

Let $X$ be a projective scheme over a field. Let $\mathscr{L}$ and $\mathscr{M}$ be invertible sheaves on $X$. Suppose that the volume of $\mathscr{L}$ exists as a limit. Then, $ \vol_{\mathscr{L}}(\mathscr{M}) $ also exists as a limit and

$$ \vol_{\mathscr{L}}(\mathscr{M}) = \vol(\mathscr{L}). $$

\end{lemma}

\begin{proof}

By \cref{{cor: invertible_sheaves_from_divisors}}, we can assume that $\mathscr{M}=\mathcal{O}_{X}(D)$ for some Cartier divisor $D$. Let us consider first the case where $D$ is effective. We have a short exact sequence 

$$ 0 \rightarrow \mathcal{O}_{X} \left( -D \right) \rightarrow \mathcal{O}_{X} \rightarrow \mathcal{O}_{D} \rightarrow 0 $$

\noindent which, after tensoring with $ \mathcal{O}_{X} \left( D \right) $, becomes

$$ 0 \rightarrow \mathcal{O}_{X} \rightarrow \mathcal{O}_{X} \left( D \right) \rightarrow \mathcal{O}_{D} \otimes \mathcal{O}_{X} \left( D \right) \rightarrow 0.$$

\noindent Now, $\vol_{\mathscr{L}}(\mathcal{O}_{X})=\vol(\mathscr{L})$ exists as a limit by assumption. Also, the sheaf $\mathcal{O}_{D} \otimes \mathcal{O}_{X} \left( D \right)$ is supported on a closed set of dimension smaller than the dimension of $X$. Thus, by \cref{lem: existence_limits_SES}, $\vol_{\mathscr{L}}(\mathcal{O}_{X}(D))$ exists as a limit and $\vol(\mathscr{L}) = \vol_{\mathscr{L}}(\mathcal{O}_{X}(D))$.

Suppose now that $D$ is an arbitrary Cartier divisor. Thanks to \cref{cor: invertible_sheaves_from_divisors}, we can write $D \sim A-B$ where both $A$ and $B$ are effective Cartier divisors. Since $B$ is effective, we have a short exact sequence

$$ 0 \rightarrow \mathcal{O}_{X} \left(-B\right) \rightarrow \mathcal{O}_{X} \rightarrow \mathcal{O}_{B} \rightarrow 0. $$

\noindent We can tensor the sequence above with $\mathcal{O}_{X} \left( A \right)$ and get 

$$ 0 \rightarrow \mathcal{O}_{X} \left(D\right) \rightarrow \mathcal{O}_{X} \left( A \right) \rightarrow \mathcal{O}_{B} \otimes \mathcal{O}_{X} \left(A \right) \rightarrow 0. $$

\noindent Since $A$ is effective, we proved above that $\vol_{\mathscr{L}}(\mathcal{O}_{X}(A))$ exists as a limit and $\vol(\mathscr{L}) = \vol_{\mathscr{L}}(\mathcal{O}_{X}(A))$. Moreover, the sheaf $\mathcal{O}_{B} \otimes \mathcal{O}_{X} \left(A \right)$ is supported on a closed set of dimension smaller than the dimension of $X$. By \cref{lem: existence_limits_SES}, we conclude that $\vol_{\mathscr{L}}(\mathcal{O}_{X}(D))$ exists as a limit and

$$ \vol_{\mathscr{L}}(\mathcal{O}_{X}(D)) = \vol_{\mathscr{L}}(\mathcal{O}_{X}(A)) = \vol(\mathscr{L}). $$

\end{proof}

\subsection{Volumes of generically free coherent sheaves with respect to line bundles}

\begin{lemma}\label{lem: coherent_sheaves on irreducible}

Let $X$ be an irreducible projective scheme over a field and let $\mathscr{L}$ be a line bundle on $X$ whose volume exists as a limit. Let $\mathcal{F}$ be a generically free coherent sheaf. Then the volume of $\mathcal{F}$ with respect to $\mathscr{L}$ exists as a limit and

$$\\vol_{\mathscr{L}}(\mathcal{F}) = \rank(\mathcal{F})\\vol(\mathscr{L}).  $$

\end{lemma}

\begin{proof}

Let $\mathscr{A}$ be an ample line bundle such that $\mathcal{F} \otimes \mathscr{A}$ is generated by global sections. Let $\xi$ be the generic point of $X$. Let $n = \rank(\mathcal{F})$. Choose $n$ global sections $s_1, \dots, s_n$ such that they form a basis of $(\mathcal{F} \otimes \mathscr{A})_\xi$ over $\mathcal{O}_{X,\xi}$. These $n$ sections induce a map $\oplus_{i=1}^n\mathcal{O}_X \rightarrow \mathcal{F} \otimes \mathcal{A}$, which is an isomorphism at the generic point. We have then a short exact sequence

$$ 0 \rightarrow \mathcal{K} \rightarrow \bigoplus_{i=1}^n \mathcal{O}_X \rightarrow \mathcal{F} \otimes \mathscr{A} \rightarrow \mathcal{C} \rightarrow 0,$$

\noindent where $\dim(\Supp(\mathcal{K})) < \dim(X)$ and $\dim(\Supp(\mathcal{C})) < \dim(X)$. Tensoring with $\mathscr{M}\coloneqq\mathscr{A}^{-1}$ we get

$$ 0 \rightarrow \mathcal{K}\otimes\mathscr{M} \rightarrow \bigoplus_{i=1}^n\mathscr{M} \rightarrow \mathcal{F} \rightarrow \mathcal{C}\otimes\mathscr{M} \rightarrow 0.$$

\noindent Notice that 

$$\\vol_\mathscr{L}\left(\bigoplus_{i=1}^{n} \mathscr{M}\right)=n\vol_{\mathscr{L}}(\mathscr{M})=\rank(\mathcal{F})\vol_{\mathscr{L}}(\mathscr{M})$$

\noindent exists as a limit by \cref{lem: tensoring_invertible_preserves_limit}. The result now follows from \cref{cor: existence_limits_isomorphism_away_from_closed_sets} applied to the second exact sequence.
\end{proof}

\section{Existence of Volumes of Linear Series on Varieties}

In this section we will state a fundamental result regarding volumes of linear series. The main theorems in this paper will depend crucially on this theorem. \cref{thm: existence_volume_linear_series_on_varieties} was proven in \cite{LazMus} for certain linear series on projective varieties over an algebraically closed field and in \cite{KavKov} for arbitrary linear series on projective varieties over algebraically closed fields. The following theorem follows from \cite[Theorem 8.1]{Cut01}.

\begin{theorem}[Cutkosky]\label{thm: existence_volume_linear_series_on_varieties}

Let $X$ be a $d$-dimensional proper variety over a field $k$. Let $L$ be a graded linear series on $X$ and let $m = m(L)$ be its index. Then,

$$ \lim \limits_{n \to \infty} \dfrac{\dim_{k}L_{mn}}{n^{d}}$$

\noindent exists.

\end{theorem}

Let $\mathscr{L}$ be a line bundle on a proper variety. It can be shown that if $\vol(\mathscr{L})>0$, then the index of $\mathscr{L}$ is equal to one (the index of a line bundle is defined in \cref{sec: Notation}). We can thus take $L=R(\mathscr{L})$ in \cref{thm: existence_volume_linear_series_on_varieties} to show that the volume of any line bundle on a proper variety exists as a limit.

\section{Nilradicals and Volumes} \label{sec: Nilradicals_and_Volumes}

The following proposition refines \cite[Proposition 4.1]{CN}. The first part of the proof is as in \cite{CN}, but we repeat it here for the reader's convenience.

\begin{proposition} \label{prop: all_depends_on _volume_of_nilradical}

Let $X$ be an irreducible projective scheme over a field $k$. Let $\mathcal{N}$ be the nilradical of $X$. Let $\mathscr{L}$ be an invertible sheaf on $X$. Suppose that $\vol_{\mathscr{L}}(\mathcal{N})$ exists as a limit. Then, $\vol(\mathscr{L})$ exists as a limit as well and we have the formula

$$ \vol(\mathscr{L}) = \vol(L) + \vol_{\mathscr{L}}(\mathcal{N}),$$

\noindent where $L$ is the linear series obtained by restricting the sections of $\mathscr{L}$ to $X_{red}$. Moreover, either $m(L)=1$ or $\vol(L)=0$.

\end{proposition}

\begin{proof}

Let $d=\dim(X)$. Fix $n \in \mathbb{N}$. The short exact sequence 

$$ 0 \rightarrow \mathcal{N} \rightarrow \mathcal{O}_{X} \rightarrow \mathcal{O}_{X_{red}} \rightarrow 0 $$

\noindent can be tensored with $\mathscr{L}^n$ to get

$$ 0 \rightarrow \mathcal{N}\otimes \mathscr{L}^{n} \rightarrow \mathscr{L}^{n} \rightarrow \left(\mathscr{L}|_{X_{red}}\right)^{n} \rightarrow 0. $$

\noindent Taking global sections and using the additivity of dimension we get that

\begin{equation}\label{eq: basic_SES}
\dfrac{h^{0}(X, \mathscr{L}^{n})}{n^d/d!} = \dfrac{h^{0}(X, \mathcal{N} \otimes \mathscr{L}^{n})}{n^d/d!} + \dfrac{\dim_{k}L_n}{n^d/d!}, 
\end{equation}

\noindent where $L_n=image\left(\coh^{0}(X,\mathscr{L}^{n}) \rightarrow \coh^{0}(X_{red}, \left(\mathscr{L}|_{X_{red}}\right)^{n})\right)$. The $k$-algebra $L= \oplus_n L_n$ is a graded linear series associated to $\mathscr{L}|_{X_{red}}$. Fix an ample divisor $\mathscr{A}$ such that there exist sections $\alpha \in \coh^{0}(X, \mathscr{A})$ and $\beta \in \coh^{0}(X, \mathscr{A} \otimes \mathscr{L})$ that are non-zero-divisors. The fact that this is possible is a consequence of \cref{lem: existence_nzd}. We have in particular that the restrictions $\alpha|_{X_{red}}$ and $\beta|_{X_{red}}$ to $X_{red}$ are not zero. This is the case since $\alpha$ and $\beta$ are non-zero-divisors and $X_{red}$ is an associated subvariety of $X$. We consider two cases.

\noindent \textit{Case 1}. Suppose that there exists $n_0 > 0$ such that the restriction map

$$ \coh^0(X, \mathscr{A}^{-1} \otimes \mathscr{L}^{n_0}) \rightarrow \coh^0(X_{red}, \mathscr{A}^{-1} \otimes \left(\mathscr{L}|_{X_{red}} \right)^{n_0} ) $$

\noindent is not zero. Choose $\gamma \in \coh^0(X, \mathscr{A}^{-1} \otimes \mathscr{L}^{n_0})$ such that $\gamma|_{X_{red}} \neq 0$. Since $X_{red}$ is a variety, 

$$0 \neq \alpha|_{X_{red}} \otimes \gamma|_{X_{red}} \in L_{n_0}$$

\noindent and 

$$0 \neq \beta|_{X_{red}} \otimes \gamma|_{X_{red}} \in L_{n_0+1}.$$ 

\noindent It follows that the index of $L$ is equal to 1. By \cref{thm: existence_volume_linear_series_on_varieties}, $\vol(L)$ exists as a limit. Recall that $\vol_{\mathscr{L}}(\mathcal{N})$ exists as a limit by assumption. Taking limits in \eqref{eq: basic_SES} one gets that $\vol(\mathscr{L})$ exists as a limit and

$$ \vol(\mathscr{L}) = \vol(L) + \vol_{\mathscr{L}}(\mathcal{N}).$$

\noindent \textit{Case 2}. Suppose now that for all $n>0$ the restriction map

$$ \coh^0(X, \mathscr{A}^{-1} \otimes \mathscr{L}^{n}) \rightarrow  \coh^0(X_{red},  \mathscr{A}^{-1} \otimes \left(\mathscr{L}|_{X_{red}} \right)^{n} ) $$

\noindent is the zero map. We begin by showing that $\vol_{\mathscr{L}}(\mathcal{N} \otimes \mathscr{A}^{-1})$ exists as a limit and 

\begin{equation} \label{eq: volume_with_equal_volume_without}
\vol_{\mathscr{L}}(\mathcal{N} \otimes \mathscr{A}^{-1}) = \vol_{\mathscr{L}}(\mathcal{N}).
\end{equation}

\noindent The global section $\alpha$ of $\mathscr{A}$ induces a short exact sequence

\begin{equation}\label{eq: hyperplane_section}
0 \rightarrow \mathscr{A}^{-1} \rightarrow \mathcal{O}_{X} \rightarrow \mathcal{O}_{H} \rightarrow 0,    
\end{equation}

\noindent where $H \coloneqq div(\alpha)$ is a closed subscheme of dimension smaller than the dimension of $X$. We can tensor this exact sequence with $\mathcal{N}$ to get 

$$ 0 \rightarrow \mathcal{K} \rightarrow \mathscr{A}^{-1} \otimes \mathcal{N} \rightarrow \mathcal{N} \rightarrow \mathcal{O}_{H} \otimes \mathcal{N} \rightarrow 0. $$

\noindent We will show that $\mathcal{K}=0$. Let $p \in X$ and let $f \in \mathcal{O}_{X,p}$ be a local equation for $\alpha$ at $p$. We have that $f$ is a non-zero-divisor on $\mathcal{O}_{X,p}$. Moreover, since $\mathcal{N}_p \subset \mathcal{O}_{X,p}$, then $f$ is a non-zero-divisor on $\mathcal{N}_{p}$. Thus, $(\mathscr{A}^{-1} \otimes \mathcal{N})_p \rightarrow \mathcal{N}_p$ is injective and $\mathcal{K}_p=0$. Since $p$ was arbitrary, $\mathcal{K}=0$.

Therefore, we have a short exact sequence 

$$ 0 \rightarrow \mathscr{A}^{-1} \otimes \mathcal{N} \rightarrow \mathcal{N} \rightarrow \mathcal{O}_{H} \otimes \mathcal{N} \rightarrow 0. $$

\noindent Since $ \vol_{\mathscr{L}}(\mathcal{N}) $ exists as a limit by assumption, we can apply \cref{lem: existence_limits_SES} to conclude that $\vol_{\mathscr{L}}(\mathcal{N} \otimes \mathscr{A}^{-1})$ exists as a limit and \eqref{eq: volume_with_equal_volume_without} holds. 

It follows from the assumption of \textit{Case 2} that 

$$ \coh^{0}(X,\mathcal{N}\otimes\mathscr{A}^{-1}\otimes \mathscr{L}^{n}) =  \coh^0(X, \mathscr{A}^{-1} \otimes \mathscr{L}^{n})$$

\noindent for all $n>0$. This shows the $\vol_{\mathscr{L}}(\mathscr{A}^{-1})$ exists as a limit and

\begin{equation}\label{eq: with_without_nilradical}
\vol_{\mathscr{L}}(\mathscr{A}^{-1})= \vol_{\mathscr{L}}(\mathcal{N} \otimes \mathscr{A}^{-1}).
\end{equation}

\noindent Finally, apply \cref{lem: existence_limits_SES} to the short exact sequence \eqref{eq: hyperplane_section} to get that $\vol(\mathscr{L})$ exists as a limit and 

\begin{equation} \label{eq: volume_of_L}
\vol(\mathscr{L}) = \vol_{\mathscr{L}}(\mathscr{A}^{-1}).
\end{equation}

\noindent Combining \eqref{eq: volume_with_equal_volume_without}, \eqref{eq: with_without_nilradical}, and \eqref{eq: volume_of_L}, we get the relation

$$  \vol(\mathscr{L}) = \vol_{\mathscr{L}}(\mathcal{N}). $$

\noindent Looking back at \eqref{eq: basic_SES}, this shows that $\vol(L)=0$.

\end{proof}

\section{Proof of \cref{thm: existence_limits_codimension1}}

As a first step towards proving \cref{thm: existence_limits_codimension1}, we address the case where $X$ is normal and the subscheme $Y$ has an invertible ideal sheaf. So, our first goal is to prove the following proposition.

\begin{proposition} \label{prop: existence_of_limits_on_invertible_ideal_sheaf_on_normal_varieties}

Let $X$ be a normal projective variety over an arbitrary field. Let $\mathcal{I}$ be a locally principal ideal sheaf with associated closed subscheme $Y$. Then, the volume of any invertible sheaf on $Y$ exists as a limit.

\end{proposition}

We will begin by proving a special case of this proposition, but before we need some preliminaries. Let $X$ be a Noetherian scheme. Let $\mathcal{P} \subset \mathcal{O}_{X}$ be a prime ideal sheaf, that is, $\mathcal{P}$ is an ideal sheaf whose associated closed subscheme $Y$ is integral. Let $\xi$ be the generic point of $Y$. We define the $n$\textsuperscript{th} symbolic power of $\mathcal{P}$, denoted $\mathcal{P}^{(n)}$, to be the ideal sheaf given as the kernel of the natural map 

$$  \mathcal{O}_{X} \rightarrow \dfrac{\mathcal{O}_{X, \xi}}{m_{\xi}^{n}}. $$

\noindent We agree that the homomorphism $\mathcal{O}_{X}(U) \rightarrow \mathcal{O}_{X, \xi}/m_{\xi}^{n}$ is the zero map whenever $\xi \notin U$. Thus, if $U \subset X$ is an open set not containing $\xi$, we have that $\mathcal{P}^{(n)}(U)=\mathcal{O}_{X}(U)$. On the other hand, for an affine open subset $U \subset X$ containing $\xi$, $\mathcal{P}^{(n)}(U)$ is equal to the $n$\textsuperscript{th} symbolic power of the prime ideal $\mathcal{P}(U)$ in the ring $\mathcal{O}_{X}(U)$.

\begin{proposition} \label{prop: existence_of_limits_on_symbolic_power_of_integral_subschme_of_projective_scheme}

Let $X$ be a normal projective variety over an arbitrary field. Let $\mathcal{P}$ be a prime ideal sheaf with associated integral subscheme $Y$. Assume that $Y$ has codimension one. Let $m$ be a positive integer. Denote by $Y^{(m)}$ the subscheme associated to the ideal sheaf $\mathcal{P}^{(m)}$. Then, the volume of any invertible sheaf on $Y^{(m)}$ exists as a limit. 

\end{proposition}

\begin{proof}

We argue by induction on $m$, where the case $m=1$ follows from \cref{thm: existence_volume_linear_series_on_varieties} (\cite[Proposition 8.1]{Cut01}) since $Y=Y^{(1)}$ is a projective variety. Fix some $m >1$. Let $\mathcal{N}$  be the nilradical of $Y^{(m)}$ and let $\mathscr{L}$ be an invertible sheaf on $Y^{(m)}$. By \cref{prop: all_depends_on _volume_of_nilradical}, since $Y^{(m)}$ is irreducible, it is enough to prove that $\vol_{\mathscr{L}}(\mathcal{N})$ exists as a limit and we now proceed to do so.

We first show that $\mathcal{N}$ is generically free of rank one as an $\mathcal{O}_{Y^{(m-1)}}$-module. Notice that $\mathcal{N}=\mathcal{P}/\mathcal{P}^{(m)}$ is a $\mathcal{O}_{Y^{(m-1)}}=\mathcal{O}_X/\mathcal{P}^{(m-1)}$-module, since $\mathcal{P} \cdot \mathcal{P}^{(m-1)} \subset \mathcal{P}^{(m)}$. Let $\xi$ be the generic point of $Y$. Since $X$ is normal and $Y$ has codimension one, $X$ must be nonsingular at $\xi$. We then have that $\mathcal{P}_{\xi} = (f)$ for some $f \in \mathcal{O}_{X,\xi}$. Since $\mathcal{P}_{\xi}$ is principal, $\mathcal{P}_{\xi}^{(m)}=\mathcal{P}_{\xi}^{m}$ and then $\mathcal{N}_{\xi} = \mathcal{P}_{\xi}/\mathcal{P}^{(m)}_{\xi} = (f)/(f^m)$ is free of rank one as a $\mathcal{O}_{Y^{(m-1)},\xi}$-module.

By induction, the volume of $\mathscr{L}|_{Y^{(m-1)}}$ exists as a limit. By \cref{lem: coherent_sheaves on irreducible}, we have that 

$$\vol_{\mathscr{L}|_{Y^{(m-1)}}}(\mathcal{N}) = \vol(\mathscr{L}|_{Y^{(m-1)}})$$

\noindent and these volumes exist as limits. Now, for any $n \in \mathbb{N}$, since $\mathcal{N}$ is both an $\mathcal{O}_{Y^{(m)}}$-module and an $\mathcal{O}_{Y^{(m-1)}}$-module, we have that 

$$ \coh^{0}( Y^{(m)}, \mathcal{N} \otimes \mathscr{L}^{n}) = \coh^{0}(Y^{(m-1)}, \mathcal{N} \otimes \left( \mathscr{L}|_{Y^{(m-1)}}\right)^n ) $$

\noindent and hence the limit

\begin{equation*}
    \begin{split}
    \vol_{\mathscr{L}}(\mathcal{N})&= \lim \limits_{n \to \infty} \dfrac{h^{0}\left( Y^{(m)}, \mathcal{N} \otimes \mathscr{L}^{n} \right)}{n^{d}/d!} = \lim \limits_{n \to \infty} \dfrac{h^{0}\left( Y^{(m-1)}, \mathcal{N} \otimes (\mathscr{L}|_{Y^{(m-1)}})^{n} \right)}{n^{d}/d!}\\
    &= \vol_{\mathscr{L}|_{Y^{(m-1)}}}\left(\mathcal{N}\right)
\end{split}
\end{equation*} 

\noindent exists, where $d=\dim(Y^{(m)}) = \dim(Y^{(m-1)})$. By \cref{prop: all_depends_on _volume_of_nilradical}, the proof is complete. 

\end{proof}

\begin{remark}

Notice that we have also shown in the proof that 

$$\vol_{\mathscr{L}}(\mathcal{N})= \vol(\mathscr{L}|_{Y^{(m-1)}}).$$

\noindent By \cref{prop: all_depends_on _volume_of_nilradical}, we obtain the relation

\begin{equation} \label{eq: volume_line_bundle_on_Y}
\vol(\mathscr{L})=\vol(\mathscr{L}|_{Y^{(m-1)}}) + \vol(L),    
\end{equation}

\noindent where $L$ is the linear series obtained by restricting the sections of $\mathscr{L}$ to $Y$. Moreover, either $m(L)=1$ or $\vol(L)=0$. 

If we assume that $\mathscr{L}$ is a nef invertible sheaf, we can refine \eqref{eq: volume_line_bundle_on_Y}. Under this additional assumption, we claim that $\vol(\mathscr{L}|_{Y}) = \vol(L)$. Indeed, let $d=\dim(Y^{(m)})$ and take cohomology of the short exact sequence

$$ 0 \rightarrow \mathcal{N} \otimes \mathscr{L}^{n} \rightarrow \mathscr{L}^n \rightarrow \mathcal{O}_{Y} \otimes \mathscr{L}^{n} \rightarrow 0$$

\noindent to obtain

$$ 0 \rightarrow \coh^{0}(Y^{(m)}, \mathcal{N} \otimes \mathscr{L}^{n}) \rightarrow \coh^{0}(Y^{(m)}, \mathscr{L}^n) \rightarrow \coh^{0}(Y,\left(\mathscr{L}|_{Y}\right)^{n}) \rightarrow V_n \rightarrow 0, $$

\noindent where $V_n$ is the image of the map $\coh^{0}(Y,\left(\mathscr{L}|_{Y}\right)^{n}) \rightarrow \coh^{1}(Y^{(m)}, \mathcal{N} \otimes \mathscr{L}^{n})$. In particular, we have the relation

\begin{equation} \label{eq}
h^{0}(Y,\left(\mathscr{L}|_{Y}\right)^{n})= h^{0}(Y^{(m)}, \mathscr{L}^{n}) - h^{0}(Y^{(m)}, \mathcal{N} \otimes \mathscr{L}^{n}) + \dim_{k}(V_{n}).
\end{equation}

\noindent Now, by \cite[Proposition 1.31]{Deb}, since $\mathscr{L}$ is nef, we have that 

$$ \lim \limits_{n \to \infty}\dfrac{\dim_{k}V_{n}}{n^d} \leq \lim \limits_{n \to \infty}\dfrac{h^{1}(Y^{(m)}, \mathcal{N} \otimes \mathscr{L}^{n})}{n^{d}}=0.  $$

\noindent Thus, dividing \eqref{eq} by $n^{d}/d!$ and letting $n$ go to infinity we see that

$$\vol(\mathscr{L}|_{Y}) = \vol(\mathscr{L}) - \vol_{\mathscr{L}}(\mathcal{N}),$$

\noindent and the quantity on the right is exactly $ \vol(L) $. Therefore, \eqref{eq: volume_line_bundle_on_Y} yields by induction the relation

$$ \vol(\mathscr{L}) = m \cdot \vol(\mathscr{L}|_{Y}). $$

\noindent It is known \cite[Section 2.2.C]{Laz01} that, when $\mathscr{L}$ is a nef invertible sheaf,

$$ \vol(\mathscr{L}) = ( \mathscr{L} )^d, $$

\noindent where $( \mathscr{L} )^d$ is the top self-intersection of $\mathscr{L}$. The equation above is then equivalent to

$$ ( \mathscr{L} )^d = m \cdot ( \mathscr{L}|_{Y})^d.$$

\end{remark}

\vspace{.2in}

Before giving the proof of \cref{prop: existence_of_limits_on_invertible_ideal_sheaf_on_normal_varieties}, for lack of a precise reference, we state here as a lemma the following standard fact.

\begin{lemma}\label{lem: principal_ideal_normal_ring}
Let $R$ be a Noetherian normal domain and let $I$ be a proper nonzero principal ideal. Then, there exist height one prime ideals $P_1, \dots, P_t$ and positive integers $n_1, \dots , n_t$ such that

$$I = P_1^{(n_1)} \cap \cdots \cap P_t^{(n_t)}.$$

\end{lemma}

\begin{proof}

We have an irredundant primary decomposition

$$ I = Q_1 \cap \dots \cap Q_t, $$

\noindent where the $Q_i$ are primary ideals. By \cite[Theorem 11.5]{Mat}, the prime ideals $P_i=\sqrt{Q_i}$ have height one. Thus, for any $i$, we can localize at $P_i$ to get

\begin{equation} \label{eq: I_equal_Q}
I_{P_i} = (Q_i)_{P_i},
\end{equation}

\noindent thanks to \cite[Proposition 4.8.(i)]{Ati}. Since $R$ is normal and $P_i$ has height one, the ring $R_{P_i}$ is a DVR. Thus, $I_{P_{i}} = (P_i)^{n_i}_{P_i}$ for some positive integer $n_i$. We can then contract \eqref{eq: I_equal_Q} back to $R$ and apply \cite[Proposition 4.8.(ii)]{Ati} to get that

$$ P_{i}^{(n_i)} = Q_i.$$

\end{proof}

\begin{proof}[Proof of \cref{prop: existence_of_limits_on_invertible_ideal_sheaf_on_normal_varieties}]

Since $X$ is normal and $\mathcal{I}$ locally principal, there exist prime ideal sheaves $\mathcal{P}_{1}, \dots, \mathcal{P}_{t}$ such that 

$$ \mathcal{I} = \mathcal{P}_{1}^{(n_1)} \cap \dots \cap \mathcal{P}_{t}^{(n_t)}.$$

\noindent Moreover, the varieties associated to the ideal sheaves $\mathcal{P}_{i}$ have codimension one in $X$. This follows from \cref{lem: principal_ideal_normal_ring} applied to an affine open cover of $X$. Let $Y_{i}^{(n_{i})}$ be the subscheme associated to the ideal sheaf $ \mathcal{P}_{i}^{(n_i)} $. This decomposition gives an exact sequence

$$ 0 \rightarrow \mathcal{O}_{Y} \rightarrow \prod_{i=1}^{t} \mathcal{O}_{Y_{i}^{(n_{i})}} \rightarrow \mathcal{C} \rightarrow 0.$$

\noindent The sheaf $\mathcal{C}$ is supported on a closed subset of dimension smaller than the dimension of $Y$. Let $\mathscr{L}$ be an invertible sheaf on $Y$. The volume of $\prod_{i=1}^{t} \mathcal{O}_{Y_{i}^{(n_{i})}}$ with respect to $\mathscr{L}$ exists as a limit by \cref{prop: existence_of_limits_on_symbolic_power_of_integral_subschme_of_projective_scheme}. By \cref{lem: existence_limits_SES}, the volume of $\mathscr{L}$ exists as a limit and

$$\vol(\mathscr{L}) = \sum_{i=1}^{t} \vol\left( \mathscr{L}|_{{Y_i}^{(n_i)}} \right).$$

\end{proof}

We are now ready to give a proof of \cref{thm: existence_limits_codimension1}. We recall the notation for convenience. In the following, $X$ will represent a variety over an arbitrary field and $Y$ will denote a subscheme of $X$ of codimension one. Moreover, we assume that no codimension one component of $Y$ is contained in the nonnormal locus of $X$.

\begin{proof}[Proof of \cref{thm: existence_limits_codimension1}]

Let $V$ be the nonsingular locus of $X$ and let $Z$ be a codimension one component of $Y$. We claim that $V \cap Z \neq \emptyset$. To see this, first let $W$ be the normal locus of $X$. By assumption, $W \cap Z $ is a closed codimension one subset of $W$. Since $W$ is a normal open subscheme of $X$, its singular locus has codimension at least two and hence $W \cap Z $ must contain a nonsingular point of $W$. But this point is also nonsingular in $X$ since $W$ is open and hence $V \cap Z \neq \emptyset$ as we claimed.

Let $\pi \colon \widetilde{X} \rightarrow X$ be the normalization of the blow-up of the ideal sheaf $\mathcal{I}_{Y}$ of $Y$ in $X$. By \cite[Lemma 6.11]{Cut_Book} (which is valid over arbitrary fields), the morphism $\pi$ is an isomorphism over the open set

$$ U\coloneqq \{ p \in X \colon p \text{ is normal and }\mathcal{I}_{Y,p} \text{ is a principal ideal} \}.$$

\noindent Notice that every codimension one component of $Y$ intersects $U$. Indeed, every component of $Y$ contains a nonsingular point of $X$, nonsingular points are normal, and the ideal sheaf $\mathcal{I}$ is principal at nonsingular points since $Y$ is of codimension one.

Let $\widetilde{Y}$ be the total transform of $Y$ in $\widetilde{X}$, that is, $\widetilde{Y}$ is the closed subscheme of $\widetilde{X}$ given by the invertible ideal sheaf $\mathcal{I}_{Y}\mathcal{O}_{\widetilde{X}}$. The morphism $\pi$ restricts to a morphism $\widetilde{Y} \rightarrow Y$, which we also call $\pi$. We have that $\pi$ is an isomorphism over $Y \cap U$. Moreover, $\pi_{*}\mathcal{O}_{\widetilde{Y}}$ is a coherent $\mathcal{O}_{Y}$-module since both $\widetilde{X}$ and $X$ are projective varieties (and hence $\pi$ is a projective morphism). We have an exact sequence

$$ 0 \rightarrow \mathcal{K} \rightarrow \mathcal{O}_{Y} \rightarrow \pi_{*}\mathcal{O}_{\widetilde{Y}} \rightarrow \mathcal{C} \rightarrow 0, $$

\noindent where the sheaves $\mathcal{K}$ and $\mathcal{C}$ are supported on closed subsets of dimension smaller than the dimension of $Y$. Let $\mathscr{L}$ be an invertible sheaf on $Y$. By the projection formula, for any $n \in \mathbb{N}$, we have that $\pi_{*}\mathcal{O}_{\widetilde{Y}} \otimes \mathscr{L}^{n} \cong \pi_{*}\left(\pi^{*}\mathscr{L} \right)^{n}$
and hence $\vol_{\mathscr{L}}(\pi_{*}\mathcal{O}_{\widetilde{Y}})$ exists as a limit if and only if $\vol(\pi^{*}\mathscr{L})$ exists as a limit. Now,  notice that $\widetilde{Y}$ is a closed subscheme of the normal projective variety $\widetilde{X}$ given by an invertible ideal sheaf. By \cref{prop: existence_of_limits_on_invertible_ideal_sheaf_on_normal_varieties}, the volume of $\pi^{*}\mathscr{L}$ exists as a limit. By \cref{cor: existence_limits_isomorphism_away_from_closed_sets} applied to the exact sequence above, the volume of $\mathscr{L}$ exists as a limit. Moreover, 

$$ \vol(\mathscr{L}) = \vol_{\mathscr{L}}(\pi_{*}\mathcal{O}_{\widetilde{Y}}) = \vol(\pi^{*}\mathscr{L}) .$$

\end{proof}

\section{Projective Schemes with $\mathcal{N}^2=0$}

\subsection{Volumes of Coherent Sheaves on Reduced Schemes}

The following proposition shows that the volume of a coherent $\mathcal{O}_{X}$-module with respect to an invertible sheaf always exists on reduced schemes. Together with \cref{prop: all_depends_on _volume_of_nilradical}, this result is the main ingredient in the proof of \cref{thm: nilradical_squared_equals_zero}  

\begin{proposition} \label{prop: invertible_times_coherent}

Let $X$ be a reduced projective scheme over an arbitrary field. Let $X_{1}, \dots, X_{r}$ be the irreducible components of $X$ of maximal dimension. Let $\mathscr{L}$ be an invertible sheaf on $X$ and let $\mathcal{F}$ be a coherent $\mathcal{O}_{X}$-module. Then, the volume of $\mathcal{F}$ with respect to $\mathscr{L}$ exists as a limit and we have the formula

$$ \vol_{\mathscr{L}}(\mathcal{F}) = \sum \limits_{j=1}^{r} \rank(\mathcal{F}|_{X_j})\vol(\mathscr{L}|_{X_j}), $$

\noindent where the volumes $\vol(\mathscr{L}|_{X_j})$ exist as limits since the $X_j$ are projective varieties.

\end{proposition}

This statement can be deduced with a little work from \cref{lem: coherent_sheaves on irreducible}.

\begin{proof}[Proof of \cref{prop: invertible_times_coherent}]

Let us assume first that $X$ is irreducible and hence a variety. By the comment after \cref{thm: existence_volume_linear_series_on_varieties}, we have that $\vol(\mathscr{L})$ exists as a limit. Let $\xi$ be the generic point of $X$. If $\mathcal{F}_\xi=0$, then $\mathcal{F}$ is supported on a closed set of dimension smaller than $\dim(X)$ and then the limit in question equals zero. Thus, given that $X$ is a variety, we can assume that $\mathcal{F}$ is generically free. By \cref{lem: coherent_sheaves on irreducible},

$$\vol_\mathscr{L}(\mathcal{F})=\rank(\mathcal{F})\vol(\mathscr{L})$$

\noindent exists as a limit.

If $X$ is not irreducible, let $\xi_1, \dots, \xi_t$ be the generic points of the irreducible components $X_1, \dots, X_t$ of $X$ and label them in such a way that, for some $r \leq t$, $X_1, \dots, X_r $ are all the components of maximal dimension. Since $X$ is reduced, we have a short exact sequence

$$ 0 \rightarrow \mathcal{O}_{X} \rightarrow \prod \limits_{j=1}^t \mathcal{O}_{X_j} \rightarrow \mathcal{C} \rightarrow 0,$$

\noindent where $ \mathcal{C} $ is supported on a closed subset of dimension less than the dimension of $X$. This exact sequence can be tensored with $\mathcal{F}$ to obtain

\begin{equation} \label{ses: kernel_non_flat}
    0 \rightarrow \mathcal{K} \rightarrow \mathcal{F}  \rightarrow \prod \limits_{j=1}^t  \mathcal{O}_{X_j} \otimes \mathcal{F} \rightarrow \mathcal{C} \otimes \mathcal{F} \rightarrow 0.
\end{equation}

\noindent We claim that $\mathcal{K}$ is supported on a closed set of dimension smaller than the dimension of $X$. Indeed, for any $j$, since $X$ is reduced, we have that $\mathcal{O}_{X_j,\xi_j} = \mathcal{O}_{X,\xi_j}$. This shows that the map $\mathcal{F} \rightarrow \prod_j  \mathcal{O}_{X_j} \otimes \mathcal{F}$ is injective at $\xi_j$. Hence, $\mathcal{K}_{\xi_j} = 0$ for all $j \in \{ 1, \dots, t \}$ and this implies that $\dim(\Supp(\mathcal{K})) < \dim(X)$. Since we also have that $\dim(\Supp(C \otimes \mathcal{F})) < \dim(X)$, we can apply \cref{cor: existence_limits_isomorphism_away_from_closed_sets} to \eqref{ses: kernel_non_flat}  and the invertible sheaf $\mathscr{L}$ to obtain that

\begin{equation*}
\begin{split}
    \vol_{\mathscr{L}}(\mathcal{F}) &= \sum \limits_{j=1}^{t} \lim \limits_{n \to \infty} \dfrac{h^{0} \left( X, \mathcal{O}_{X_j} \otimes \mathcal{F} \otimes \mathscr{L}^{n} \right)}{n^d/d!} \\
&=\sum \limits_{j=1}^{t} \lim \limits_{n \to \infty} \dfrac{h^{0} \left( X_j, \mathcal{F}|_{X_{j}} \otimes \left( \mathscr{L}|_{X_j} \right)^{n} \right)}{n^d/d!},
\end{split}
\end{equation*}

\noindent where the limits in the second equality exist by the case treated above since the $X_j$ are projective varieties. Now, the limit

$$\lim \limits_{n \to \infty} \dfrac{h^{0} \left( X_j, \mathcal{F}|_{X_{j}} \otimes \left( \mathscr{L}|_{X_j} \right)^{n} \right)}{n^d/d!}$$

\noindent is equal to zero if $\dim(X_j) < d$, by \cite[Proposition 1.31]{Deb}, and equal to $\vol_{\mathscr{L}|_{X_j}}(\mathcal{F}|_{X_j})$ whenever $\dim(X_j) = d$. Thus, we can consider only the components of maximal dimension and obtain the formula

$$  \vol_{\mathscr{L}}(\mathcal{F}) = \sum \limits_{j=1}^{r} \vol_{\mathscr{L}|_{X_j}}(\mathcal{F}|_{X_j}) = \sum \limits_{j=1}^{r} \rank(\mathcal{F}|_{X_j})\vol(\mathscr{L}|_{X_j}). $$

\end{proof}

\subsection{Proof of \cref{thm: nilradical_squared_equals_zero}}

We recall for convenience the notation of \cref{thm: nilradical_squared_equals_zero}. In what follows, $X$ will denote a $d$-dimensional projective scheme over a field $k$. The nilradical of $X$ will be represented by $\mathcal{N}=\mathcal{N}(X)$ and $\mathscr{L}$ will denote an invertible sheaf on $X$. We suppose that $\mathcal{N}^2=0$.

\begin{proof}[Proof of \cref{thm: nilradical_squared_equals_zero}]

Let us consider first the case where $X$ is irreducible. Let $X_{red}$ be the closed subscheme of $X$ defined by the ideal sheaf $\mathcal{N}$. Notice that $X_{red}$ is a variety. In order to apply \cref{prop: all_depends_on _volume_of_nilradical}, we have to show that $\vol_{\mathscr{L}}(\mathcal{N} )$ exists as a limit. Since $\mathcal{N}^2=0$, we have that $\mathcal{N} \subset Ann_{\mathcal{O}_X}(\mathcal{N})$ and, in consequence, $\mathcal{N}$ has the structure of an $\mathcal{O}_{X_{red}}$-module. It follows that

$$ \coh^{0}(X, \mathcal{N} \otimes \mathscr{L}^{n} ) = \coh^{0}(X_{red}, \mathcal{N}|_{X_{red}} \otimes \left(\mathscr{L}|_{X_{red}}\right)^{n}). $$

\noindent Thus, the limit

$$\vol_{\mathscr{L}}(\mathcal{N})=\lim \limits_{n \to \infty} \dfrac{h^{0}(X_{red},  \mathcal{N|}_{X_{red}} \otimes \left(\mathscr{L}|_{X_{red}}\right)^{n})}{n^d/d!} $$

\noindent exists by \cref{prop: invertible_times_coherent} with $\mathcal{F}=\mathcal{N}|_{X_{red}}$. By \cref{prop: all_depends_on _volume_of_nilradical}, $\vol(\mathscr{L})$ exists as a limit.

If $X$ is not irreducible, let $0 = \mathcal{Q}_1 \cap \cdots \cap \mathcal{Q}_{t} $ be an irredundant primary decomposition of the zero ideal sheaf with respect to a given projective embedding (this can be obtained by sheafifying a primary decomposition of the zero ideal in the coordinate ring of $X$ with respect to the given embedding). Let $Y_i$ be the irreducible closed subscheme of $X$ determined by the ideal sheaf $\mathcal{Q}_i$. We have a short exact sequence

$$0 \rightarrow \mathcal{O}_{X} \rightarrow \prod_{i=1}^{t} \mathcal{O}_{Y_{i}} \rightarrow \mathcal{C} \rightarrow 0,$$

\noindent where the cokernel $\mathcal{C}$ is supported on a closed set of smaller dimension than $X$. By \cref{lem: existence_limits_SES}, in order to show that $\vol(\mathscr{L})$ exists as a limit, it is enough to show that the limits

$$ \lim \limits_{n \to \infty} \dfrac{h^0 \left(X, \mathcal{O}_{Y_i} \otimes \mathscr{L}^n \right)}{n^d}=\lim \limits_{n \to \infty} \dfrac{h^0 \left(Y_i, \left( \mathscr{L}|_{Y_i}\right)^n \right)}{n^d} $$

\noindent exist for all $i$. By \cite[Proposition 1.31]{Deb}, these limits will be equal to zero whenever $Y_i$ is not a subscheme of maximal dimension. Thus, let us assume that $Y_i$ has maximal dimension $d$. Notice that $Y_i$ is irreducible. By the case considered at the beginning of the proof, it is enough to show that the nilradical $\mathcal{N}(Y_i)$ of $Y_i$ satisfies $\mathcal{N}(Y_i)^{2} = 0$. We prove this by working locally.

Let $U=Spec(R)$ be an affine open subset of $X$. If $U \cap Y_i =\emptyset$, then $\mathcal{N}(Y_i)^{2}|_{U} = 0$ trivially. We can hence assume that $U \cap Y_i \neq \emptyset$. Moreover, to simplify the notation, suppose that $U \cap Y_j \neq \emptyset$ for all other indices $j$ as well. For otherwise, the ideal sheaves $\mathcal{Q}_j|_{U} = \mathcal{O}_{U}$ and we can disregard them in our considerations. We have $\mathcal{Q}_j = \widetilde{Q_j}$ for ideals $Q_j$ of $R$. Let $P_j = \sqrt{Q_j}$. Then the ideals $Q_j$ are $P_j$-primary. We have that $\mathcal{N}(X)|_U=\widetilde{N}$, where $N$ is the nilradical of $R$ and by assumption $N^2=0$. Also notice that $\mathcal{O}_{Y_i}|_{U} = \widetilde{R/Q_i}$ and $\mathcal{N}(Y_i)|_{U} = \widetilde{P_i/Q_i}$. Thus, it is enough to show that $P_i^2 \subset Q_i$ and we now proceed to do so.

In $R$ we have the following relation:

$$ Q_1 \cap \cdots \cap Q_t = 0 = N^2 = (P_1 \cap \cdots \cap P_t       )^2. $$

\noindent Since $Y_i$ has maximal dimension, $P_i$ is a minimal prime. We can then localize the equation above at $P_i$ and, with the help of \cite[Proposition 4.8.(i)]{Ati}, obtain 

$$ (Q_i)_{P_i} = (P_i)_{P_i}^2 = (P_i^2)_{P_i}.$$

\noindent By \cite[Proposition 4.8.(ii)]{Ati}, $Q_i = (P_i^2)_{P_i} \cap R \supset P_i^2$. Therefore, we get that $ \mathcal{N}(Y_i)^2|_{U}= \widetilde{\left(P_i/Q_i\right)^2} = \widetilde{(P_i^2+Q_i)/Q_i}=0$. Since $U$ was arbitrary, $\mathcal{N}(Y_i)^2=0$ and the proof is complete.

\end{proof}

\begin{remark}

As in the previous proposition, let $X$ be a projective scheme over a field with nilradical $\mathcal{N}$. Suppose that $\mathcal{N}^{2}=0$. Assume for simplicity that $X$ is irreducible. Let $\mathscr{L}$ be in invertible sheaf on $X$. \cref{thm: nilradical_squared_equals_zero} shows that the volume of $\mathscr{L}$ exists as a limit. We now give a formula for this volume.

The nilradical $\mathcal{N}$ is an $\mathcal{O}_{X_{red}}$-module as explained in the proof of \cref{thm: nilradical_squared_equals_zero}. Then 

\begin{equation}\label{eq: remark1}
\vol_{\mathscr{L}}(\mathcal{N}) = \vol_{\mathscr{L}|_{X_{red}}}(\mathcal{N}|_{X_{red}})= \rank(\mathcal{N}|_{X_{red}})\vol(\mathscr{L}|_{X_{red}}),
\end{equation}

\noindent where the second equality follows form \cref{prop: invertible_times_coherent}. By \cref{prop: all_depends_on _volume_of_nilradical}, we conclude that

$$ \vol(\mathscr{L}) = \vol(L) + \rank(\mathcal{N}|_{X_{red}})\vol(\mathscr{L}|_{X_{red}}), $$

\noindent where $L$ is the linear series obtained by restricting the sections of $\mathscr{L}$ to $X_{red}$. 

\noindent If $\mathscr{L}$ is a nef invertible sheaf, as shown in the remark after \cref{prop: existence_of_limits_on_symbolic_power_of_integral_subschme_of_projective_scheme}, $\vol(L) = \vol(\mathscr{L}|_{X_{red}})$ and we get the relation

$$ \vol(\mathscr{L}) = \left(\rank(\mathcal{N}|_{X_{red}}) + 1 \right)\vol(\mathscr{L}|_{X_{red}}).$$

\end{remark}

\end{document}